\documentclass[20pt, oneside]{article}   	
\usepackage{geometry}                		
\usepackage{amsmath}
\usepackage{amsthm}	
\usepackage{amssymb}
\usepackage{tikz}
\usepackage[bottom]{footmisc}
\usepackage[affil-it]{authblk}
\usepackage[toc,page]{appendix}
\newtheorem{thm}{Theorem}[section]
\newtheorem{defi}[thm]{Definition}

\newtheorem{prop}[thm]{Proposition}

\newtheorem{lem}[thm]{Lemma}
\newtheorem{cor}[thm]{Corollary}
\newtheorem*{claim*}{Claim}

\newtheorem*{thm*}{Theorem}

\newcommand{\F}{\mathbb{F}}
\newcommand{\fl}{\mathfrak{l}}

\newcommand{\Gal}{\mathrm{Gal}}

\newcommand{\GL}{\mathrm{GL}}

\newcommand{\xdownarrow}[1]{%
  {\left\downarrow\vbox to #1{}\right.\kern-\nulldelimiterspace}
}

\title{ Masser-W\"ustholz bound for reducibility of Galois representations for Drinfeld modules of arbitrary rank}
\author{Chien-Hua Chen \thanks{Electronic address: \texttt{danny30814@ncts.ntu.edu.tw}; ORCID: \texttt{0000-0003-3267-5603} ; Corresponding author}}
\affil{Mathematics Division,\\ National Center for Theoretical Sciences,\\ Taipei, Taiwan}

\begin{document}

\maketitle
\abstract
In this paper, we give an explicit bound on the irreducibility of mod-$\fl$ Galois representation for Drinfeld modules of arbitrary rank without complex multiplication. This is a function field analogue of Masser-W\"ustholz bound on irreducibility of mod-$\ell$ Galois representation for elliptic curves over number field.

\section {Introduction}
In 1993, Masser and W\"ustholz \cite{MW93-2} proved a famous result on existence of isogeny, with degree bounded by an explicit formula,  between two isogenous Elliptic curves. Building upon this achievement, they \cite{MW93} subsequently employed the isogeny estimation to establish an explicit bound on the irreducibility of mod-$\ell$ Galois representation associated to elliptic curves over number field without complex multiplication (CM).  This bound is then used to deduce a bound on the surjectivity of mod-$\ell$ Galois representation for elliptic curves over number field without CM. 

Analogous to the elliptic curve theory, David and Denis \cite{DD99} introduced an isogeny estimation applicable to Anderson $t$-modules. In particular, they deduced an isogeny estimation for Drinfeld $\F_q[T]$-modules over a global function field, see Theorem \ref{dd} for more details. This naturally prompts the query of whether the strategy employed by Masser-W\"ustholz can be adapted to deduce an irreducibility limit for mod-$\fl$ Galois representations concerning rank-$r$ Drinfeld modules without CM. However, the Masser-W\"ustholz strategy can not be applied directly to the context of Drinfeld modules. The main resean is when one computes the degree of isogeny between Drinfeld modules, the degree is always a power of $q$, which is not a prime number. Thus the computational trick in Lemma 3.1 of \cite{MW93} does not work for Drinfeld modules.

Nevertheless, the fundamental concept underlying the Masser-W\"ustholz approach has inspired us to produce a similar method. By combining this concept with the height estimation on isogenous Drinfeld modules, as established by Breuer, Pazuki, and Razafinjatovo (as detailed in Theorem \ref{bp}), we are equipped to deduce our main result: an explicit bound on the irreducibility of the mod-$\fl$ Galois representation for Drinfeld modules of any rank, without CM.

\begin{thm}\label{main}
Let $q=p^e$ be a prime power, $A:=\F_q[T]$, and $K$ be a finite extension of $F:=\F_q(T)$ of degree $d$. Let $\phi$ be a rank-$r$ Drinfeld $A$-module over $K$ of generic characteristic and assume that ${\rm End}_{\bar{K}}(\phi)=A$. Let $\fl=(\ell)$ be a prime ideal of $A$, consider the mod-$\fl$ Galois representation
$$\bar{\rho}_{\phi,\fl}:\Gal(\bar{K}/K)\rightarrow {\rm Aut}(\phi[\fl])\cong \GL_r(A/\fl).$$
If $\bar{\rho}_{\phi,\fl}$ is reducible, then either
\begin{equation}
\deg_T\ell-N_d\log \deg_T\ell   \leqslant \log C+ N_d\left\{ \log d +r+\log[h_G(\phi)+2]               \right\}
\end{equation}

or
\begin{equation}
\deg_T\ell\leqslant \log C+N_d[\log d\cdot h(\phi)]
\end{equation}

Here  $C$ is a computable constant depending on $q$ and $r$, and $N_d=10(d+1)^7$. Furthermore, $h(\phi)$ denotes the naive height of Drinfeld module, while $h_G(\phi)$ is the graded height of Drinfeld module. We refer their explicit definitions to Definition \ref{height}.
\end{thm}
As a corollary of Theorem \ref{main}, we deduce a sufficient condition for $\deg_T\ell$ to have irreducible mod-$\fl$ Galois representation $\bar{\rho}_{\phi,\fl}$. See corollary \ref{cormain} for details.

Regarding the specialized scenario of "rank-$2$ Drinfeld modules over $\F_q(T)$," a more nuanced estimation concerning the irreducibility of the mod-$\fl$ Galois representation has been advanced by Chen and Lee \cite{CL19}. However, their strategy uses the fact that a power of $1$-dimensional group representation is again a group representation, see proof of Proposition 7.1 in \cite{CL19}. In the context of rank-2 scenarios, the reducibility of mod-$\fl$ Galois representation always contributes a $1$-dimensional subrepresentation. But this is not true for higher rank Drinfeld modules.

On the other hand, Chen and Lee \cite{CL19} gave an explicit bound on surjectivity of mod-$\fl$ Galois representations for rank-$2$ Drinfeld modules over $\F_q(T)$ without CM. Such an explicit bound is still unknown for higher rank Drinfeld modules. The main difficulty is the classification of maximal subgroups (up to conjugacy classes) in $\GL_r$ over finite field is much more complicated comparing to the $\GL_2$ case, where one only need to take care of the Borel and Cartan cases.

\section{Preliminaries}

Let $A=\F_q[T]$ be the polynomial ring over finite field with $q=p^e$ an odd prime power, $F=\F_q(T)$ be the fractional field of $A$, and  $K$ be a finite extension over $F$. Throughout this paper, ``$\log$'' refers to the logarithm with base $q$.

\subsection{Drinfeld modules}

\begin{defi}

Let $K<x>:=\left\{\sum_{i=0}^{n}c_ix^{q^i}\mid c_i\in K\right\}$. Define $(K<x>,+,\circ)$ to be the ring of twisted $q$-polynomials with usual addition, and the multiplication is defined to be composition of $q$-polynomials.

\end{defi}

\begin{defi}
A Drinfeld $A$-module of rank $r$ over $K$ of generic characteristic is a ring homomorphism

$$\phi: A\rightarrow K<x>, \ \ a\mapsto \phi_a(x) $$
determined by
$$\phi_T(x)=Tx+g_1x^q+\cdots+g_rx^{q^r}.$$
\end{defi}

For an ideal $\mathfrak{a}=<a>$ of $A$, we may define the $\mathfrak{a}$-torsion of the Drinfeld module $\phi$ over $K$.

\begin{defi}
The $\mathfrak{a}$-torsion of a Drinfeld module $\phi$ over $K$ is defined to be
$$\phi[\mathfrak{a}]:=\left\{\textrm{ zeros of }\phi_a(x) \textrm{ in } \bar{K} \right\}\subset \bar{K}.$$
\end{defi}

Now we define the $A$-module structure on $\bar{K}$. For any elements $b\in A$ and $\alpha\in \bar{K}$. We define the $A$-action of $b$ on $\alpha$ via
$$b\cdot\alpha:=\phi_b(\alpha).$$
This gives $\bar{K}$ an $A$-module structure. And the $A$-module structure inherits to $\phi[\mathfrak{a}]$. As our Drinfeld module $\phi$ over $K$ has generic characteristic, we have the following proposition

\begin{prop}
$\phi[\mathfrak{a}] $ is a free $A/\mathfrak{a}$-module of rank $r$.
\end{prop}
\begin{proof}
See \cite{G96}, Proposition 4.5.3.
\end{proof}

Let $\fl$ be a prime ideal of $A$, then the $\fl$-torsion $\phi[\fl]$ of the Drinfeld module $\phi$ is an $r$-dimensional $A/\fl$-vector space. Applying the action of absolute Galois group $\Gal(\bar{K}/K)$ on $\phi[\fl]$, we obtain the so-called mod-$\fl$ Galois representation
$$\bar{\rho}_{\phi,\fl}:\Gal(\bar{K}/K)\rightarrow {\rm Aut}(\phi[\fl])\cong \GL_r(A/\fl)$$
 for the Drinfeld module $\phi$ over $K$.

Now we proceed to define various heights associated with Drinfeld modules. Let $M_K$ be the set of all places of $K$ including places above $\infty$. For each place $\nu\in M_K$, we define $n_\nu:=[K_\nu:F_\nu]$ to be the degree of local field extension $K_\nu/F_\nu$. Furthermore, we set $|\cdot|_\nu$ to be a normalized valuation of $K_\nu$. 

\begin{defi}\label{height}
Let $\phi$ be a rank-$r$ Drinfeld module over $K$ characterized by
$$\phi_T(x)=Tx+g_1x^q+\cdots+g_{r-1}x^{q^{r-1}}+g_rx^{q^r}, \textrm{ where } g_i\in K \textrm{ and }g_r\in K^*.$$

\begin{enumerate}
\item The naive height of $\phi$ is defined to be 
$$h(\phi):={\rm max}\{h(g_1),\cdots,h(g_r)\},$$
where $h(g_i):=\frac{1}{[K:F]}\sum_{\nu\in M_K}n_\nu\cdot {\rm log}|g_i|_\nu$.

\item The graded height of $\phi$ is defined to be
$$h_G(\phi):=\frac{1}{[K:F]}\sum_{\nu\in M_K}n_\nu\cdot {\rm log}\ {\rm max}\{|g_i|_\nu^{1/(q^i-1)}\mid 1\leqslant i\leqslant r\}$$

\end{enumerate}
\end{defi}

\begin{cor}\label{heightineq}
One can observe from the definition of naive height and graded height that
$$h(\phi)\leqslant(q^r-1)\cdot h_G(\phi).$$

\end{cor}

\subsection{Isogenies}
\begin{defi}
Let $\phi$ and $\psi$ be two rank-$r$ Drinfeld $A$-modules over $K$. A {\bf{morphism}} $u:\phi\rightarrow \psi$ over $K$ is a twisted $q$-polynomial $u\in K<x>$ such that
$$u\phi_a=\psi_a u \text{ \rm for all } a\in A.$$ 
A non-zero morphism $u:\phi\rightarrow \psi$ is called an isogeny. A morphism $u:\phi\rightarrow \psi$ is called an {\bf{isomorphism}} if its inverse exists. 
\end{defi}

Set ${\rm Hom}_K(\phi,\psi)$ to be the group of all morphisms $u:\phi\rightarrow \psi$ over $K$. We denote ${\rm End}_K(\phi)={\rm Hom}_K(\phi,\phi)$. For any field extension $L/K$, we define
$${\rm Hom}_L(\phi,\psi)=\{u\in L<x> \mid u\phi_a=\psi_a u  \text{ \rm for all } a\in A \}.$$
For $L=\bar{K}$, we omit subscripts and write

$${\rm Hom}(\phi,\psi):={\rm Hom}_{\bar{K}}(\phi,\psi) \text{ \rm and } {\rm End}(\phi):={\rm End}_{\bar{K}}(\phi)$$

\begin{defi}
The composition of morphisms makes ${\rm End}_L(\phi)$ into a subring of $L<x>$, called the {\bf{endomorphism ring}} of $\phi$ over $L$. For any rank-$r$ Drinfeld module $\phi$ over $K$ with ${\rm End}(\phi)=A$, we say that $\phi$ does not have complex multiplication.

\end{defi}

\begin{defi}
Let $f:\phi\rightarrow \psi$ be an isogeny of Drinfeld modules over $K$ of rank $r$, we define the degree of $f$ to be
$$\deg f:=\# {\rm{ker}}(f).$$

\end{defi}

\begin{prop}\label{dual}
Let $f:\phi\rightarrow \psi$ be an isogeny of Drinfeld modules over $K$ of rank $r$. There exists a dual isogeny $\hat{f}:\psi\rightarrow \phi$ such that $$f\circ\hat{f}=\psi_a \textrm{ and } \hat{f}\circ f=\phi_a.$$
Here  $0\neq a\in A$ is an element of minimal $T$-degree such that ${\rm ker}(f)\subset \phi[a]$. 

\end{prop}
\begin{proof}
See \cite{G96} Proposition 4.7.13 and Corollary 4.7.14.

\end{proof}

The following corollary is immediate by counting cardinalities.  
\begin{cor}\label{countdeg}
As in the setting of Proposition \ref{dual}, we have 
$$q^{r\cdot\deg_T(a)}=(\deg f)\cdot(\deg \hat{f}).$$

\end{cor}

Now we can state the key tools to derive our main result:

\begin{thm}[\cite{DD99} Theorem 1.3]\label{dd}
 Let $K$ be a finite extension over $F$ with $[K:F]:=d$. Suppose that there are two $\bar{K}$-isogenous Drinfeld modules $\phi$ and $\psi$ defined over $K$. Then there is an isogeny $f: \phi\rightarrow \psi$ such that
 $$\deg f\leqslant c_2\cdot (dh(\phi))^{10(r+1)^7}.$$

Here $c_2=c_2(r,q)$ is a effectively computable constant depends only on $r$ and $q$.

\end{thm}

\begin{thm}[\cite{BPR21} Theorem 3.1]\label{bp}
Let $f:\phi\rightarrow \psi$ be an isogeny of rank-$r$ Drinfeld modules over $\bar{K}$ and suppose that ${\rm ker}(f)\subset \phi[N]$ for some $0\neq N\in A$. Then we have

$$|h_G(\psi)-h_G(\phi)|\leqslant \deg_T(N)+\left( \frac{q}{q-1}-\frac{q^r}{q^r-1}\right).$$

\end{thm}

\section{Proof of Theorem \ref{main}}
We are given a rank-$r$ Drinfeld module $\phi$ defined over $K$ with ${\rm End}(\phi)=A$. Suppose the image of  mod-$\fl$ Galois representation ${\rm{Im}}\bar{\rho}_{\phi,\fl}$ acting on $\phi[\fl]$ has an invariant $A/\fl$-subspace of dimension $1\leqslant k\leqslant r-1$. Denote such an invariant subspace by $H$.

From Proposition 4.7.11 and Remark 4.7.12 of \cite{G96},  there is an isogeny $$f:\phi\rightarrow \phi/H$$ with ${\rm ker}(f)=H$. Since $\phi$ and $f$ both are defined over $K$, one can see that the Drinfeld module $\phi/H$ is a rank-$r$ Drinfeld module defined over $K$ as well. In addition, we have $$\deg f=\# H=q^{k\cdot \deg_T\fl}.$$

Take a dual isogeny $\hat{f}:\phi/H\rightarrow \phi$ of $f$. The degree of $\hat{f}$ can be computed using Corollary \ref{countdeg}. We get
$$\deg \hat{f}=q^{(r-k)\cdot\deg_T\fl}.$$

Besides, we can find two isogenies between $\phi$ and $\phi/H$ with bounded degree from Theorem \ref{dd}:
\begin{enumerate}
\item[$\bullet$]  $u: \phi\rightarrow \phi/H$ is such an isogeny defined over $\bar{K}$ with $\deg u\leqslant c_2\cdot(dh(\phi))^{10(d+1)^7}$

\item[$\bullet$]$u': \phi/H\rightarrow \phi$ is such an isogeny defined over $\bar{K}$ with $\deg u'\leqslant c_2\cdot(dh(\phi/H))^{10(d+1)^7}$\end{enumerate}

Since ${\rm End}(\phi)=A$, we have $u'\circ u=\phi_b$ for some $b\in A$.

Now we consider the composition of isogenies
$$u'\circ f\circ \hat{f}\circ u:\phi\rightarrow \phi/H\rightarrow \phi\rightarrow \phi/H\rightarrow \phi.$$

Since ${\rm End}(\phi)=A$, we can find  $N_1$ and $N_2$ in $A$ such that 
$$u'\circ f=\phi_{N_1}, \textrm{ and }\hat{f}\circ u=\phi_{N_2}.$$
Thus we have 
$$u'\circ f\circ \hat{f}\circ u=(u'\circ f)\circ (\hat{f}\circ u)=\phi_{N_1N_2}.$$
On the other hand, we compute in different order and get
$$u'\circ f\circ \hat{f}\circ u=u'\circ (f\circ \hat{f})\circ u=u'\circ(\phi/H)_\ell\circ u=\phi_\ell\circ(u'\circ u)=\phi_{\ell b}.$$
Thus we get the equality $\ell b=N_1N_2$.  As $\ell$ is prime, we have either case (1): $\ell | N_1$ or case (2): $\ell | N_2$.

\begin{enumerate}
\item[case (1):] $\ell | N_1$.

Then we may write $N_1=\ell\cdot \beta$ for some $0\neq\beta\in A$. From the equality $u'\circ f=\phi_{N_1}$, we have

$$\log \deg u'+k\cdot \deg_T\ell=r(\deg_T\ell+\deg_T\beta).$$

Hence we get $$\log \deg u'= (r-k)\deg_T\ell+r\deg_T\beta.$$ Combining with the bound $\deg u'\leqslant c_2\cdot(dh(\phi/H))^{10(d+1)^7}$, we obtain the inequality
$$(r-k)\deg_T\ell\leqslant \log c_2 +10(d+1)^7\log[dh(\phi/H)]-r\deg_T\beta\leqslant \log c_2 +10(d+1)^7\log[dh(\phi/H)].$$

Thus we have
\begin{equation}
\deg_T\ell\leqslant\frac{1}{r-k}\cdot \left(  \log c_2 +10(d+1)^7\log[dh(\phi/H)]  \right)\leqslant \log c_2 +10(d+1)^7\log[dh(\phi/H)].\tag{$\star$}
\end{equation}

Now from Corollary \ref{heightineq} and Theorem \ref{bp}, we have
$$h(\phi/H)\leqslant (q^r-1)h_G(\phi/H)\leqslant (q^r-1)\cdot\left[h_G(\phi)+\deg_T\ell+(\frac{q}{q-1}-\frac{q^r}{q^r-1})       \right].$$
Deduce from the above inequality, we get
$$\begin{array}{ccc}
\log h(\phi/H)&\leqslant& \log (q^r-1)+\log \left( h_G(\phi)+\deg_T\ell+(\frac{q}{q-1}-\frac{q^r}{q^r-1})                    \right)\\
\ \\
&\leqslant&r+\log \deg_T\ell+\log\left(      h_G(\phi)+1+(\frac{q}{q-1}-\frac{q^r}{q^r-1})               \right)
\end{array}$$

Combining with the inequality ($\star$) and the fact that $$\frac{q}{q-1}-\frac{q^r}{q^r-1}<1 ,$$ we have the inequality
$$
\deg_T\ell-10(d+1)^7\log \deg_T\ell   \leqslant \log c_2+ 10(d+1)^7\left\{ \log d +r+\log[h_G(\phi)+2]               \right\}.
$$
After renaming $C:=c_2$ and $N_d:= 10(d+1)^7$, we get the desired inequality (1).

\item[case (2):] $\ell | N_2$.

Then we may write $N_2=\ell\cdot \beta$ for some $0\neq\beta\in A$. From the equality $\hat{f}\circ u=\phi_{N_2}$, we have
$$(r-k)\deg_T\ell+\log \deg u=r(\deg_T\ell+\deg_T\beta).$$

Thus we get $\log \deg u=k\deg_T\ell+r\deg_T \beta$. Together with the bound $\deg u\leqslant c_2\cdot(dh(\phi))^{10(d+1)^7}$, we achieve that
$$k\deg_T\ell\leqslant \log c_2+10(d+1)^7\log [dh(\phi)]-r\deg_T\beta\leqslant c_2\cdot(dh(\phi))^{10(d+1)^7}.$$
Hence we have the inequality
$$
\deg_T\ell\leqslant \frac{1}{k}\cdot \left(\log c_2+10(d+1)^7[\log d\cdot h(\phi)]\right)\leqslant \log c_2+10(d+1)^7[\log d\cdot h(\phi)].
$$
Again, relabelling $C:=c_2$ and $N_d:= 10(d+1)^7$ gives us the inequality (2).

\end{enumerate}
This complete the proof of Theorem \ref{main}. 
\section{Lower bound on irreducibility of  $\bar{\rho}_{\phi,\fl}$}

Under the setting of Theorem \ref{main}, one may further solve the inequality (1) for $\deg_T\ell$. By setting
$$\Omega_\phi:={\rm max}\left\{\log C+ N_d\left( \log d +r+\log[h_G(\phi)+2]\right), \log C+N_d[\log d\cdot h(\phi)] \right\},$$
Theorem \ref{main} implies that  the mod-$\fl$ Galois representation is irreducible when $$\textrm{(1'): } \frac{q^{\deg_T\ell}}{\deg_T\ell^{N_d}}>q^{\Omega_\phi}  \textrm{ and (2): } \deg_T\ell>\Omega_\phi$$

Since when we fix a finite extension $K/F$ and a Drinfeld module $\phi$, the numbers $N_d$ and $\Omega_\phi$ are fixed. Elementary Calculus can tell us that the fraction $\frac{q^{\deg_T\ell}}{\deg_T\ell^{N_d}}$ tends to infinity as $\deg_T\ell$ goes to infinity. Thus we can always find a real number $C_{\phi,d}$ such that $\deg_T\ell>C_{\phi.d}$ implies $\frac{q^{\deg_T\ell}}{\deg_T\ell^{N_d}}>q^{\Omega_\phi}$. Now we try to compute $C_{\phi,d}$ explicitly:

\begin{lem}\label{num}

Let $a, b,$ and $c$ be positive real numbers such that $c^{1/b}\cdot\frac{b}{{\rm ln}a}\geqslant e$, where $e$ is the Euler's number and ${\ln(\cdot)}:={\rm log}_e(\cdot)$. Then 
$$x>\frac{-b\cdot W_{-1}(\frac{-{\rm ln}a}{c^{1/b}\cdot b})}{{\rm ln}\ a}$$
is a solution to the inequality $$\frac{a^x}{x^b}>c.$$

Here $W_{-1}$ is the negative brach of the real-valued Lambert $W$-function, i.e. the inverse function of the complex valued function $f(y)=ye^y$.

\end{lem}

\begin{proof}
The proof is done by direct computation, hence we leave it to the reader as an exercise. 
\end{proof}

 Now we take $x=\deg_T\ell$, $a=q$, $b=N_d$, and $c=q^{\Omega_\phi}$. One can check that 
 $$c^{1/b}\cdot\frac{b}{{\rm ln}(a)}\geqslant e.$$
 Therefore, Lemma \ref{num} shows that
 $$C_{\phi,d}=\frac{-b\cdot W_{-1}(\frac{-{\rm ln}(a)}{c^{1/b}\cdot b})}{{\rm ln} (a)}.$$
And we can conclude the following corollary:

\begin{cor} \label{cormain}
Let $q=p^e$ be a prime power, $A:=\F_q[T]$, and $K$ be a finite extension of $F:=\F_q(T)$ of degree $d$. Let $\phi$ be a rank-$r$ Drinfeld $A$-module over $K$ of generic characteristic and assume that ${\rm End}_{\bar{K}}(\phi)=A$. Let $\fl=(\ell)$ be a prime ideal of $A$, consider the mod-$\fl$ Galois representation
$$\bar{\rho}_{\phi,\fl}:\Gal(\bar{K}/K)\rightarrow {\rm Aut}(\phi[\fl])\cong \GL_r(A/\fl).$$

If $\deg_T\ell>{\rm max}\{C_{\phi,d}, \Omega_\phi\}$, then $\bar{\rho}_{\phi,\fl}$ is irreducible. Here

$$\Omega_\phi:={\rm max}\left\{\log C+ N_d\left( \log d +r+\log[h_G(\phi)+2]\right), \log C+N_d[\log d\cdot h(\phi)] \right\},$$

$$C_{\phi,d}=\frac{-N_d\cdot W_{-1}(\frac{-{\rm ln}(q)}{q^{\Omega_\phi/N_d}\cdot N_d})}{{\rm ln}(q)},$$ 

and $$ C=c_2(r,q) \textrm{ is an effectively computable constant in Theorem \ref{dd}}.$$
\end{cor}

\section*{Acknowledgement}

The author would like to thank Sophie Marcques for inspiring discussions, Wei-Hung Su for showing him the computations for Lemma \ref{num}, and the referee for helpful suggestions.

\bibliographystyle{alpha}
\bibliography{DMirredbd.bib}

\end{document}